\input ./style/arxiv-vmsta.cfg
\documentclass[numbers,compress,v1.0.1]{vmsta}

\volume{4}
\issue{3}
\pubyear{2017}
\firstpage{233}
\lastpage{252}
\doi{10.15559/17-VMSTA85}

\usepackage{vtexurl}

\newtheorem{theorem}{Theorem}[section]

\newtheorem{proposition}[theorem]{Proposition}

\theoremstyle{definition}
\newtheorem{definition}[theorem]{Definition}

\newtheorem{example}[theorem]{Example}

\theoremstyle{remark}

\startlocaldefs
\newcommand{\lleft}{\left}
\newcommand{\rright}{\right}
\allowdisplaybreaks

\def\rA{{\rm A}}
\def\rB{{\rm B}}
\def\rL{{\rm L}}
\def\ttM{{\tt M}}
\def\re{{\rm e}}

\def\bt{\mathbf t}
\def\bR{\mathbf R}
\def\bE{\mathbf E}
\def\bp{\mathbf p}
\def\B1{\mathbf1}
\def\bu{\mathbf u}
\def\utheta{\underline\theta}

\def\cX{{\mathcal X}}

\def\fX{{\mathfrak X}}

\def\tM{\tilde M}
\def\BI{\mathbf I}

\def\rw{{\rm w}} \def\rNo{{\rm{No}}}

\def\BC{{\mathbf C}} \def\lam{\lambda}
\def\BD{{\mathbf D}}
\def\BX{{\mathbf X}}
\def\BY{{\mathbf Y}}
\def\tJ{\mathbf{J}}

\def\gam{\gamma}

\def\fX{\mathfrak X}
\def\cX{\mathcal X}
\def\bu{\mathbf u}
\def\bbP{\mathbb P} \def\hw{h^{\rm w}}

\def\BX{\mathbf{X}}\def\BY{\mathbf{Y}}
\def\BZ{\mathbf{Z}} \def\BN{\mathbf{N}}

\def\tM{\tilde M}
\def\bE{\mathbb{E}}
\def\bbE{{\mathbb E}} \def\bu{\mathbf u}
\def\BC{{\mathbf C}} \def\lam{\lambda}
\def\bbR{{\mathbb R}}\def\bbS{{\mathbb S}}
\def\bPhi{{\mbox{\boldmath${\varPhi}$}}}
\def\bbZ{{\mathbb Z}} \def\bt{\mathbf t}\def\B1{\mathbf1}

\def\bV{\mathbf V}
\def\mR{\cal R}
\def\bS{\mathbf S}

\def\bE{\mathbb{E}}
\def\bbE{{\mathbb E}} \def\bu{\mathbf u}
\def\BC{{\mathbf C}} \def\lam{\lambda}
\def\bbR{{\mathbb R}}\def\bbS{{\mathbb S}}

\def\bPhi{{\mbox{\boldmath${\varPhi}$}}}
\def\bbZ{{\mathbb Z}} \def\bt{\mathbf t}\def\B1{\mathbf1}

\def\rw{{\rm w}}
\def\rT{{\rm T}}
\def\tI{{\tt I}} \def\tJ{{\tt J}}
 \def\tP{{\tt P}}
\def\fX{\mathfrak X}
\def\cX{\mathcal X}
\def\bu{\mathbf u}
\def\bbP{\mathbb P} \def\hw{{h^{\rm w}}}

\def\BX{\mathbf{X}}
\def\bE{\mathbb{E}}
\def\bbE{{\mathbb E}}
\def\fX{\mathfrak X}
\def\cX{\mathcal X}
\def\bu{\mathbf u}
\def\bbP{\mathbb P} \def\hw{h^{\rm w}}

\def\rd{{\rm d}}

\def\BX{\mathbf{X}}\def\BY{\mathbf{Y}}

\def\bE{\mathbb{E}}
\def\bbE{{\mathbb E}} \def\bu{\mathbf u}
\def\BC{{\mathbf C}} \def\lam{\lambda}
\def\bbR{{\mathbb R}}\def\bbS{{\mathbb S}}
\def\bPhi{{\mbox{\boldmath${\varPhi}$}}}
\def\bbZ{{\mathbb Z}} \def\bt{\mathbf t}\def\B1{\mathbf1}
\def\bbS{\mathbb{S}}
\def\bu{\mathbf{u}}

\def\vp{\varphi}

\def\hwphi{h^{\rm w}_{\phi}}

\def\utheta{{\underline\theta}}
\let\phi\vp

\def\B1{\mathbf1}

\def\BSigma{\mathbf\varSigma}

\numberwithin{equation}{section}

\newcommand{\rrvert}{\vert}
\newcommand{\llvert}{\vert}

\endlocaldefs

\begin{document}
\begin{frontmatter}

\title{Weighted entropy: basic inequalities}

\author[a]{\inits{M.}\fnm{Mark}\snm{Kelbert}\corref{cor1}}\email{mkelbert@hse.ru}
\cortext[cor1]{Corresponding author.}
\address[a]{Higher School of Economics, Moscow, RF}

\author[b]{\inits{I.}\fnm{Izabella}\snm{Stuhl}}\email{stuhlizabella@googlemail.com}
\address[b]{Math Dept, Penn State University, PA, USA}

\author[c,d]{\inits{Y.}\fnm{Yuri}\snm{Suhov}}\email{yms@statslab.cam.ac.uk}
\address[c]{Cambridge University, UK}
\address[d]{Penn State University, PA, USA}

\markboth{M. Kelbert et al.}{Weighted entropy: basic inequalities}

\begin{abstract}
This paper represents an extended version of an earlier note [10]. The
concept of weighted entropy takes into account values of different
outcomes, i.e., makes entropy context-dependent, through the weight
function. We analyse analogs of the Fisher information inequality and
entropy power inequality for the weighted entropy and discuss connections
with weighted Lieb's splitting inequality. The concepts of rates of the
weighted entropy and information are also discussed.
\end{abstract}


\received{30 August 2017}
\revised{18 September 2017}
\accepted{18 September 2017}
\publishedonline{2 October 2017}
\end{frontmatter}

\section{Introduction}\label{Intro}
This paper represents an extended version of an earlier note
[10].\footnote{AMCTA-2017, Talks at Conference ``Analytical and
Computational Methods
in Probability Theory and its Applications''}
We also follow earlier publications discussing related topics: \cite
{SSSK,SSK,SS1,SS}.
The Shannon entropy (SE) of a probability distribution $\bp$ or the
Shannon differential entropy (SDE) of a probability density function
(PDF) $f$
\begin{equation}
h(\bp)=-\sum_i p(x_i)\log p(x_i),\qquad
h (f)=-\int f(x)\log f(x)\rd \xch{x}{x,} \label{1}
\end{equation}
is context-free, i.e., does not depend on the nature of outcomes $x_i$
or $x$,
but only upon probabilities $p(x_i)$ or values $f(x)$. It gives the
notion of entropy a great flexibility which
explains its successful applications.
However, in many situations it seems insufficient, and the context-free property
appears as a drawback. Viz., suppose you learn a news about severe
weather conditions in an area
far away from your place. Such conditions usually do not happen; an
event like this has a small probability $p\ll1$ and
conveys a high information $-\log p$. At the same time you hear that a
tree near your parking lot in the town has fallen
and damaged a number of cars. The probability of this event is also
low, so the amount of information is again high. However, the
value of this information for you is higher than in the first event.
Considerations of this character can motivate a study of weighted
information and entropy, making them context-dependent.

\begin{definition} Let us define
the {\it weighted entropy} (WE) as
\begin{equation}
h^\rw_{\phi}(\bp)=-\sum_i
\phi(x_i)p(x_i)\log p(x_i). \label{2}
\end{equation}
Here a non-negative {\it weight function} (WF) $x_i\mapsto\phi(x_i)$
is introduced, representing a value/utility of
an outcomes $x_i$. A similar approach
can be used for the differential entropy of a probability density
function (PDF) $f$. Define the
{\it weighted differential entropy} (WDE) as
\begin{equation}
h^\rw_{\phi}(f)=-\int\phi(\mathbf x)f(\mathbf x)\log f(
\mathbf x) {\rd}\mathbf x. \label{3}
\end{equation}

An initial example of a WF $\phi$ may be
$\phi(\mathbf x)={\bf1}$ $(\mathbf x\in A)$ where $A$ is a particular
subset of
outcomes (an event).
A heuristic use of the WE with such a WF was demonstrated in \cite{FS1,FS2}.
Another example repeatedly used below is
$f(\mathbf x)=f^{\rNo}_C(\mathbf x)$, a $d$-dimensional Gaussian
PDF with mean $\mathbf0$ and covariance matrix $C$. Here
\begin{align}
h^\rw_{\phi}\bigl(f^{\rNo}_C\bigr)&=\frac{\alpha_{\phi}(C)}{2}\log \bigl[(2\pi )^d{\rm det}(C) \bigr]+
\frac{\log e}{2}{\rm tr} \bigl[C^{-1}\varPhi _{C,\phi
} \bigr]\quad \hbox{where}\nonumber
\\
\alpha_{\phi}(C)&=\int_{\bbR^d}\phi(\mathbf
x)f_C^{\rNo
}(\mathbf x){\rd}\mathbf x,\quad
\varPhi_{C,\phi}=\int_{\bbR^d}\mathbf x\mathbf
x^{\rm T} \phi(\mathbf x)f_C^{\rNo}(\mathbf x){\rd}
\mathbf x.\label{eq:WDEGa}
\end{align}
For $\phi(\mathbf x)=1$ we get the normal SDE $h(f^\rNo_C)=\frac
{1}{2}\log
 [(2\pi\re)^d{\rm{det}}\,C ]$.
\end{definition}

In this note we give a brief introduction into the concept of the
weighted entropy.
We do not always give proofs, referring the reader to the quoted
original papers. Some basic properties of WE and WDE have
been presented in \cite{SSSK}; see also references therein
to early works on the subject.
Applications of the WE and WDE to the security quantification of
information systems
are discussed in \cite{PMPC}.
Other domains range from the stock market to the image processing, see,
e.g., \cite{G,KSS,LKP,NH,T,YYPL}.

Throughout this note we assume that the series
and integrals in \eqref{2}--\eqref{3} and the subsequent equations
converge absolutely, without stressing it every time again. To unify
the presentation, we
will often use integrals $\int_\cX\rd\mu$ relative to a reference
$\sigma$-finite measure $\mu$ on a
Polish space $\cX$ with a Borel $\sigma$-algebra $\fX$. In this regard,
the acronym PM/DF (probability
mass/density function) will be employed. Usual measurability
assumptions will also be
in place for the rest of the presentation. We also assume that the WF
$\phi>0$ on an open set in $\cX$.

In some parts of the presentation, the sums and integrals comprising a
PM/DF will be written as expectations:
this will make it easier to explain/use assumptions and properties
involved. Viz., Eqns \eqref{2}--\eqref{3}
can be given as $h^\rw_\phi(\bp)=-\bbE\phi(\mathbf X)\log\bp
(\mathbf X)$ and
$h^\rw_\phi(f)=-\bbE\phi(\mathbf X)\log f(\mathbf X)$
where $\mathbf X$ is a random variable (RV) with the PM/DF $\bp$ or $f$.
Similarly, in \eqref{eq:WDEGa}, $\alpha_\phi
(C)=\bbE\phi(\mathbf X)$ and $\bPhi_{C,\phi}=\bbE\phi(\mathbf
X)\mathbf X\mathbf X^\rT$
where $\mathbf X\sim{}$N$(\mathbf0,C)$.

\section{The weighted Gibbs inequality}\label{Gibbs}
Given two non-negative functions $f,g$ (typically, PM/DFs), define the
{\it weighted Kull\-back-Leibler
divergence} (or the relative WE, briefly RWE) as
\begin{equation}
D^\rw_{\phi}(f\|g)=\int_\cX\phi(\mathbf
x)f(\mathbf x)\log\frac
{f(\mathbf x)}{g(\mathbf x)}{\rd }\mu(\mathbf x). \label{5}
\end{equation}
Theorem~1.3 from \cite{SSSK} states:
\begin{theorem}
Suppose that
\begin{equation}
\int_\cX\phi(\mathbf x) \bigl[f(\mathbf x)-g(\mathbf x)
\bigr]{\rm d}\mu (\mathbf x)\geq0. \label{6}
\end{equation}
Then $D^\rw_{\phi}(f\|g)\geq0$. Moreover, $D^\rw_{\phi}(f\|g)= 0$ iff
$\phi(\mathbf x) [\frac{g(\mathbf x)}{f(\mathbf x)}-1
]=0$ $f$-a.e.
\end{theorem}

\begin{example} For an exponential family in the canonical form
\begin{equation}
f_{\utheta}(\mathbf x)=h(\mathbf x)\exp{ \bigl( \bigl\langle\utheta, T(
\mathbf x) \bigr\rangle -A(\utheta) \bigr)},\quad
\mathbf x\in\bbR^d,~\utheta \in\bbR^m, \label{5a}
\end{equation}
with the sufficient statistics $T(\mathbf x)$ we have
\begin{equation}
D^\rw_{\phi}(f_{\utheta_1}\|f_{\utheta_2})=e^{A_{\phi}(\utheta
_1)-A(\utheta_1)}
\xch{\bigl(A(\utheta_2)-A(\utheta_1)- \bigl\langle\nabla A_{\phi}(\utheta_1),\utheta_2-\utheta_1 \bigr\rangle \bigr),}%
{\bigl(A(\utheta_2)-A(\utheta_1)- \bigl\langle\nabla A_{\phi}(\utheta_1),\utheta_2-\utheta_1 \bigr\rangle \bigr)}
\label{5b}
\end{equation}
where $\nabla$ stands for the gradient w.r.t. to the parameter vector
$\utheta$, and
\begin{equation}
A_{\phi}(\utheta)=\log\int\phi(\mathbf x)h(\mathbf x)\exp \bigl( \bigl
\langle \utheta,T(\mathbf x ) \bigr\rangle\bigl)\rd\mathbf x. \label{5c}
\end{equation}
\end{example}

\section{Concavity/convexity of the weighted entropy}\label{Conc}

Theorems~2.1 and 2.2 from \cite{SSSK} offer the following assertion:
\begin{theorem}
\textup{(a)} The WE/WDE functional $f\mapsto h^\rw_{\phi}(f)$ is
concave in argument $f$. Namely,
for any PM/DFs
$f_1(x)$, $f_2(x)$ and $\lambda_1, \lambda_2\in[0,1]$ such that
$\lambda
_1+\lambda_2=1$,
\begin{equation}
h^\rw_{\phi}(\lambda_1f_1+
\lambda_2f_2)\geq\lambda_1h^\rw_{\phi
}(f_1)+
\lambda_2h^\rw_{\phi}(f_2).
\label{71}
\end{equation}
The equality iff $\phi(x)[f_1(x)-f_2(x)]=0$ holds for $(\lambda
_1f_1+\lambda
_2f_2)$-a.a. $x$.

\textup{(b)} However, the RWE functional $(f,g)\mapsto D^\rw_{\phi
}(f\|g)$ is convex: given two pairs of PDFs $(f_1,f_2)$ and $(g_1,g_2)$,
\begin{equation}
\lambda_1D^\rw_{\phi}(f_1
\|g_1)+\lambda_2D^\rw_{\phi
}(f_2
\|g_2)\geq D^\rw_{\phi}(\lambda_1f_1+
\lambda_2f_2\|\lambda_1g_1+
\lambda _2g_2),\label{72}
\end{equation}
with equality iff $\lambda_1\lambda_2=0$ or $\phi
(x)[f_1(x)-f_2(x)]=\phi(x)[g_1(x)-g_2(x)]=0$ $\mu$-a.e.
\end{theorem}

\section{Weighted Ky-Fan and Hadamard inequalities}\label{KyFan}

The map $C\mapsto\delta(C):=\log{\rm det}(C)$ gives a concave
function of a (strictly)
positive-definite $(d\times d)$ matrix $C$: $
\delta(C)-\lambda_1\delta(C_1)-\lambda_2\delta(C_2)\geq0$,
where $C=\lambda_1C_1+\lambda_2C_2$, $\lambda_1+\lambda_2=1$ and
$\lambda_{1,2}\geq0$.
This is the well-known Ky-Fan inequality. It terms of differential
entropies it is equivalent to the bound
\begin{equation}
\label{eq11} h \bigl(f^{\rNo}_{C} \bigr)-
\lambda_1h \bigl(f^{\rNo}_{C_1} \bigr)-
\lambda_2h \bigl(f^{\rNo
}_{C_2} \bigr)\geq0
\end{equation}
and is closely related to a maximising property of the Gaussian
differential entropy $h(f^\rNo_C)$.

Theorem~\ref{thm4.1} below presents one of new bounds of Ky-Fan type, in
its most explicit form, for the WF $\phi(\mathbf x)=\exp{
(\mathbf x^T\bt )},
\bt\in\bbR^d$.
Cf. Theorem~3.5 from \cite{SSSK}.
In this case the identity $h^\rw_{\phi}(f^{\rNo})=\exp{ (\frac
{1}{2}\bt^TC\bt )}h(f^{\rNo})$ holds true.
Introduce a set
\begin{equation}
{\cal S}= \bigl\{\bt\in\bbR^d:F^{(1)}(\bt)\geq0,\
F^{(2)}(\bt)\leq0 \bigr\}. \label{9}
\end{equation}
Here functions $F^{(1)}$ and $F^{(2)}$ incorporate parameters $C_i$ and
$\lam_i$:
\begin{align}
F^{(1)}(\bt)&=\sum_{i=1}^2
\lambda_i\exp{ \biggl(\frac{1}{2}\bt^TC_i
\bt \biggr)}-\exp{ \biggl(\frac{1}{2}\bt^TC\bt \biggr)},\quad \bt\in\bbR^d,
\nonumber
\\
F^{(2)}(\bt)&= \Biggl[\sum_{i=1}^2
\lambda_i\exp{ \biggl(\frac{1}{2}\bt^TC_i
\bt \biggr)}-\exp{ \biggl(\frac{1}{2}\bt^TC\bt \biggr)} \Biggr]
\log \bigl[(2\pi)^d{\rm det}(C) \bigr]
\nonumber
\\
&\quad +\sum_{i=1}^2\lambda_i
\exp{ \biggl(\frac{1}{2}\bt^TC_i\bt \biggr)} {\rm
tr} \bigl[C^{-1}C_i \bigr]-d\exp{ \biggl(\frac{1}{2}
\bt^TC\bt \biggr)},\quad \bt\in\bbR^d. \label{10}
\end{align}

\begin{theorem}\label{thm4.1}
Given positive-definite matrices $C_1,C_2$ and $\lambda_1,\lambda
_2\in
[0,1]$ with $\lambda_1+\lambda_2=1$, set $C=\lambda_1C_1+\lambda_2C_2$.
Assume $\bt\in{\cal S}$. Then
\begin{equation}
h\bigl(f_C^{\rNo}\bigr)\exp{ \biggl(\frac{1}{2}
\bt^TC\bt \biggr)}-h\bigl(f^{\rNo
}_{C_1}\bigr) \exp {
\biggl(\frac{1}{2}\bt^TC_1\bt \biggr)}-h
\bigl(f^{\rNo}_{C_2}\bigr)\exp{ \biggl(\frac
{1}{2}
\bt^TC_2\bt \biggr)}\geq0, \label{11}
\end{equation}
with equality iff $\lambda_1\lambda_2=0$ or $C_1=C_2$.
\end{theorem}

For $\bt=\mathbf0$ we obtain $\phi\equiv1$, and \eqref{11}
coincides with
\eqref{eq11}. Theorem~\ref{thm4.1} is related
to the maximisation property of the weighted Gaussian entropy which
takes the form of Theorem~\ref{thm4.2}.
Cf. Example~3.2 in \cite{SSSK}.

\begin{theorem}\label{thm4.2}
Let $f(\mathbf x)$ be a PDF on $\bbR^d$ with mean $\mathbf0$ and
$(d\times d)$ covariance
matrix~$C$. Let $f^\rNo(\mathbf x)$ stand for the Gaussian PDF, again with
the mean $\mathbf0$ and covariance matrix
$C$. Define $(d\times d)$ matrices
\begin{equation}
\label{eqbPhi}
\bPhi=\int_{\bbR^d}\phi(\mathbf x)\mathbf x\mathbf x^{\rT}f(\mathbf x)\rd \mathbf x,\qquad
\bPhi^\rNo_C=\int_{\bbR^d}\phi(\mathbf x)\mathbf x\mathbf x^{\rT}f^\rNo_C(\mathbf x)\rd\mathbf x.
\end{equation}
Cf. \eqref{eq:WDEGa}. Assume that
\[
\int_{\bbR^d}\phi(\mathbf x) \bigl[f(\mathbf
x)-f^\rNo_C(\mathbf x) \bigr]\rd\mathbf x\geq0
\]
and
\[
\log \bigl[(2\pi)^d ({\rm{det}}\,C ) \bigr]\int_{\bbR^d}
\phi (\mathbf x ) \bigl[f(\mathbf x)-f^\rNo_C(\mathbf x)
\bigr]\rd\mathbf x +{\rm{tr}} \bigl[C^{-1} \bigl(\bPhi^\rNo_C-
\bPhi \bigr) \bigr]\leq0.
\]
Then $h^\rw_\phi(f)\leq h^\rw_\phi(f^\rNo_C)$, with equality iff
$\phi(\mathbf x) [f(\mathbf x)-f^\rNo_C(\mathbf x) ]=0$
a.e.
\end{theorem}

Theorems~\ref{thm4.1} and~\ref{thm4.2} are a part of a series of the so-called weighted
{\it determinantal} inequalities. See
\cite{SSSK,SSS}. Here we will focus on a weighted version of {\it
Hadamard} inequality asserting that
for a $(d\times d)$ positive-definite matrix $C=(C_{\mathit{ij}})$,
${\rm{det}}\;C\leq\prod_{j=1}^dC_{\mathit{jj}}$ or $\delta(C)\leq
\sum_{j=1}^d\log C_{\mathit{jj}}$. Cf.
\cite{SSSK}, Theorem~3.7. Let $f^\rNo_{C_{\mathit{jj}}}$ stand for the Gaussian
PDF on $\bbR$ with the zero mean
and the variance $C_{\mathit{jj}}$. Set:
\[
\alpha=\int_{\bbR^d}\phi(\mathbf x)f^\rNo_C(
\mathbf x)\rd\mathbf x\;\hbox{ (cf. \eqref {eq:WDEGa}).}
\]

\begin{theorem}
Assume that
\[
\int_{\bbR^d}\phi(\mathbf x) \Biggl[f^\rNo_C(
\mathbf x)-\prod_{j=1}^df^\rNo
_{C_{\mathit{jj}}}(x_j) \Biggr]\rd\mathbf x\geq0.
\]
Then, with the matrix $\bPhi=(\varPhi_{\mathit{ij}})$ as in \eqref{eqbPhi},
\begin{align*}
\alpha\log\prod_{j=1}^d (2\pi C_{\mathit{jj}})
\,{+}\,(\log\re)\sum_{j=1}^dC_{\mathit{jj}}^{-1}
\varPhi_{\mathit{jj}} \,{-}\,\alpha\log \bigl[(2\pi)^d({\rm{det}}\,C)
\bigr]\,{-}\,(\log\re){\rm{tr}}\,C^{-1}\bPhi\,{\geq}\,0.
\end{align*}
\end{theorem}

\section{A weighted Fisher information matrix}\label{Fisher}
Let $\mathbf X=(X_1,\ldots,X_d)$ be a random $(1\times d)$ vector with PDF
$f_{\utheta}(\mathbf x)=f_{\mathbf X}(\mathbf x,\utheta)$ where
$\utheta=(\theta_1,\ldots
, \theta_m)\in\bbR^m$.
Suppose that $\utheta\to f_{\utheta}$ is $C^1$.
Define a score vector $S(\mathbf X,\utheta)=\B1(f_{\utheta}(\mathbf
x)>0)(\frac
{\partial}{\partial\theta_i}\log f_{\utheta}(\mathbf x), i=1,\ldots,m)$.
The $m\times m$ weighted Fisher information matrix (WFIM) is defined as
\begin{equation}
J_{\phi}^\rw(f_{\utheta})=J_{\phi}^\rw(
\mathbf X, \utheta)=\bE \bigl[\phi(\mathbf X )S(\mathbf X;\utheta)S^T(
\mathbf X;\utheta) \bigr]. \label{50}
\end{equation}
\begin{theorem}[Connection between WFIM and weighted KL-divergence measures]
For
smooth families $\{f_{\theta}, \theta\in\varTheta\in\bbR^1\}$ and a
given WF $\phi$, we get
\begin{align}
D^\rw_{\phi}(f_{\theta_1}\|f_{\theta_2})&=
\frac{1}{2}J^\rw_{\phi}(X,\theta_1) (
\theta_2-\theta_1)^2+\bE_{\theta_1}\bigl[
\phi(X)D_{\theta} \log f_{\theta_1}(X)\bigr](\theta_1-
\theta_2)
\nonumber
\\
&\quad -\frac{1}{2}\bE_{\theta_1}\biggl[\phi(X)\frac{D_{\theta}^2 f_{\theta_1}(X)}{f_{\theta_1}(X)}
\biggr](\theta_2-\theta_1)^2+o\bigl(\llvert
\theta_1-\theta_2\rrvert^2\bigr)
\bE_{\theta_1}\xch{\bigl[\phi(X)\bigr]}{\bigl[\phi(X)\bigr].} \label{111}
\end{align}
where $D_{\theta}$ stands for $\frac{\partial}{\partial\theta}$.
\end{theorem}

\begin{proof}
By virtue of a Taylor expansion of $\log f_{\theta_2}$ around $\theta
_1$, we obtain
\begin{equation}
\log f_{\theta_2}=\log f_{\theta_1}+ D_{\theta}\log
f_{\theta
_1}(\theta _2-\theta_1)+
\frac{1}{2}D_{\theta}^2\log f_{\theta_1}(\theta
_2-\theta_1)^2 +O_x \bigl(\llvert
\utheta_2-\utheta_1\rrvert^3 \bigr).
\label{112}
\end{equation}
Here $O_x(|\theta_2-\theta_1|^3)$
denotes the reminder term which has a hidden dependence on~$x$. Multiply
both sides of (\ref{112}) by $\phi$ and take expectations assuming that
we can interchange
differentiation and expectation appropriately. Next, observe that
\begin{equation}
D_{\theta}^2\log f_{\theta_1}=\frac{D_{\theta}^2f_{\theta
_1}}{f_{\theta
_1}}-
\frac{(D_{\theta}f_{\theta_1})^2}{f_{\theta_1}^2}.
\end{equation}
Hence
\begin{equation}
\bE_{\theta_1} \bigl[\phi D_{\theta}^2\log
f_{\theta_1} \bigr]=\bE _{\theta
_1} \biggl[\phi\frac{D_{\theta}^2f_{\theta_1}}{f_{\theta_1}}
\biggr]-J^\rw _{\phi}(f_{\theta_1}).
\end{equation}
Therefore the claimed result, i.e., (\ref{111}), is achieved.
\end{proof}

\section{Weighted entropy power inequality}\label{WEPI}
Let $\mathbf X_1,\mathbf X_2$ be independent RVs with PDFs $f_1,f_2$
and $\mathbf X=\mathbf X
_1+\mathbf X_2$. The famous Shannon entropy power
inequality (EPI) states that
\begin{equation}
h(\mathbf X_1+\mathbf X_2)\geq
h\xch{(\mathbf N_1+\mathbf N_2),}{(\mathbf N_1+\mathbf N_2)}
\label{11a}
\end{equation}
where $\mathbf N_1, \mathbf N_2$ are Gaussian N$(\mathbf0,\sigma
^2\mathbf I_d)$ RVs such that
$h(\mathbf X_i)=h(\mathbf N_i), i=1,2$. Equivalently,
\begin{equation}
e^{\frac{2}{d}h(\mathbf X_1+\mathbf X_2)}\geq e^{\frac{2}{d}h(\mathbf
X_1)}+\xch{e^{\frac{2}{d}h(\mathbf X_2)},}{e^{\frac{2}{d}h(\mathbf X_2)}}
\label{11b}
\end{equation}
see, e.g., \cite{CT,KS1}. The EPI is widely used in electronics, i.e.,
consider a RV $\mathbf Y$ which satisfies
\begin{equation}
\mathbf Y_n=\sum_{i=0}^{\infty} a_i\mathbf X_{n-i}, n\in \mathbf Z^1, \quad
\sum_{i=0}^{\infty}\llvert a_i\rrvert^2<\infty,
\end{equation}
where $a_i\in\bR^1$, $\{\mathbf X_i\}$ are IID RVs. Then the EPI means
\begin{equation}
h(\mathbf Y)\geq h(\mathbf X)+\frac{1}{2}
\log \xch{\Biggl(\sum_{i=0}^{\infty} \llvert a_i\rrvert^2\Biggr),}{\Biggl(\sum_{i=0}^{\infty} \llvert a_i\rrvert^2\Biggr)}
\end{equation}
with equality if and only if either $\mathbf X$ is Gaussian or if
$\mathbf Y_n=\mathbf X_{n-k}$, for some $k$, that is, the filtering operation
is a pure delay.
Clearly, a possible extension of the EPI gives more flexibility in
signal processing.
We are interested in the weighted entropy
power inequality (WEPI)
\begin{equation}
\kappa:=\exp{ \biggl(\frac{2h^\rw_{\phi}(\mathbf X_1)}{d\bE\phi
(\mathbf X_1)} \biggr)}+\exp{ \biggl(\frac{2h^\rw_{\phi}(\mathbf X_2)}{d\bE\phi(\mathbf
X_2)}
\biggr)}\leq\exp { \biggl(\frac{2h^\rw_{\phi}(\mathbf X)}{d\bE\phi(\mathbf X)} \biggr)}. \label{12}
\end{equation}
Note that (\ref{12}) coincides with (\ref{11b}) when $\phi\equiv1$.
Let $d=1$, we set
\begin{equation}
\alpha={\rm tan}^{-1} \biggl[\exp{ \biggl(\frac{h^\rw_{\phi}(X_2)}{\bE\phi(X_2)}-\frac{h^\rw_{\phi}(X_1)}{\bE\phi(X_1)} \biggr)} \biggr],
\quad
Y_1=\frac{X_1}{\cos\alpha},
Y_2=\frac{X_2}{\sin\alpha}. \label{13a}
\end{equation}

\begin{theorem}
Given independent RVs $X_1,X_2\in\bbR^1$ with PDFs $f_1,f_2$, and the
weight function $\phi$, set $X=X_1+X_2$.
Assume the following conditions:

\noindent(i)\vspace*{-12pt}
\begin{align}
\bE\phi(X_i)&\geq\bE\phi(X)\quad \mbox{if }\kappa\geq1,\ i=1,2,
\nonumber
\\
\bE\phi(X_i)&\leq\bE\phi(X)\quad \mbox{if } \kappa\leq1,\ i=1,2.
\label{14}
\end{align}
(ii) With $Y_1,Y_2$ and $\alpha$ as defined in (\ref{13a}),
\begin{equation}
(\cos\alpha)^2h^\rw_{\phi_c}(Y_1)+(\sin
\alpha)^2h^\rw_{\phi
_s}(Y_2)\leq
h^\rw_{\phi}(X), \label{15}
\end{equation}
where $\phi_c(x)=\phi(x\cos\alpha), \phi_s(x)=\phi(x\sin\alpha
)$ and
\begin{equation}
h^\rw_{\phi_c}(Y_1)=-\bE \bigl[\phi_c(Y_1)\log \bigl(f_{Y_1}(Y_1)\bigr) \bigr],
\qquad
h^\rw _{\phi_s}(Y_2)=-\bE \bigl[\phi_s(Y_2)\log \bigl(f_{Y_2}(Y_2)\bigr) \bigr].
\label{16}
\end{equation}
Then the WEPI holds.
\end{theorem}

Paying homage to \cite{L} we call (\ref{15}) weighted Lieb's splitting
inequality (WLSI).
In some cases the WLSI may be effectively checked.
\begin{proof}
Note that
\begin{align}
h^\rw_{\phi}(X_1)&=h^\rw_{\phi_c}(Y_1)+
\bE\phi(X_1)\log\cos\alpha,
\nonumber
\\
h^\rw_{\phi}(X_2)&=h^\rw_{\phi_s}(Y_2)+
\bE\phi(X_2)\log\sin\alpha. \label{31}
\end{align}
Using (\ref{15}), we have the following inequality
\begin{align}
h^\rw_{\phi}(X)&\geq(\cos\alpha)^2
\bigl[h^\rw_{\phi}(X_1)-\bE\phi(X_1)
\log\cos\alpha \bigr]
\nonumber
\\
&\quad +(\sin\alpha)^2 \bigl[h^\rw_{\phi}(X_2)-
\bE\phi(X_2)\log\sin\alpha \xch{\bigr].}{\bigr]} \label{32}
\end{align}
Furthermore, recalling the definition of $\kappa$ in (\ref{12}) we obtain
\begin{equation}
h^\rw_{\phi}(X)\geq\frac{1}{2\kappa} \bigl[\bE
\phi(X_1)\log \kappa \bigr]\exp{ \biggl(\frac{2h^\rw_{\phi}(X_1)}{\bE\phi(X_1)} \biggr)}+
\frac{1}{2\kappa} \bigl[\bE\phi(X_1)\log\kappa \bigr]\exp{ \biggl(
\frac{2h^\rw
_{\phi}(X_2)}{\bE\phi(X_2)} \biggr)}. \label{33}
\end{equation}
By virtue of assumption (\ref{14}), we derive
\begin{equation}
h^\rw_{\phi}(X)\geq\frac{1}{2}\bE\phi(X)\log\kappa.
\label{34}
\end{equation}
The definition of $\kappa$ in (\ref{12}) leads directly to the result.
\end{proof}

\begin{example} Let $d=1$ and $X_1\sim{}$N$(0,\sigma_1^2)$, $X_2\sim
{}$N$(0,\sigma_2^2)$. Then the WLSI (\ref{15}) takes the following form
\begin{align}
&\log \bigl[2\pi\bigl(\sigma_1^2+\sigma_2^2
\bigr) \bigr]\bE\phi(X)+\frac{\log e}{\sigma_1^2+\sigma_2^2}\bE\bigl[X^2\phi(X)\bigr]
\nonumber
\\
&\quad \geq(\cos\alpha)^2 \biggl[\log \biggl(\frac{2\pi\sigma_1^2}{(\cos\alpha)^2}
\biggr) \biggr]\bE\phi(X_1)+\frac{(\cos\alpha)^2\log e}{\sigma_1^2}\bE\bigl[X_1^2
\phi(X_1)\bigr]
\nonumber
\\
&\qquad +(\sin\alpha)^2 \biggl[\log \biggl(\frac{2\pi\sigma_2^2}{(\sin\alpha)^2}\biggr)
\biggr]\bE\phi(X_2)+\frac{(\sin\alpha)^2\log e}{\sigma_2^2}\bE\bigl[X_2^2
\phi(X_2)\bigr]. \label{17}
\end{align}
\end{example}

\begin{example} Let $d=1$, $X=X_1+X_2$, $X_1\sim{}$U$[a_1,b_1]$ and
$X_2\sim{}$U$[a_2,b_2]$ be independent.
Denote by $\varPhi(x)=\int_0^x \phi(u){\rm d}u$ and $L_i=b_i-a_i,\ i=1,2$.
The WDE $h^\rw_{\phi}(X_i)=\frac{\varPhi(b_i)-\varPhi(a_i)}{L_i}\log
L_i$. Then
the inequality $\kappa\geq(\leq)1$ takes the form
$L_1^2+L_2^2\geq(\leq) 1$. Suppose for definiteness that $L_2\geq L_1$
or, equivalently, $C_1:=a_2+b_1\leq a_1+b_2=:C_2$.
Inequalities (\ref{14}) take the form
\begin{equation}
L_2 \bigl[\varPhi(b_1)-\varPhi(a_1) \bigr],
\qquad
L_1 \bigl[\varPhi(b_2)-\varPhi(a_2) \bigr]\geq( \leq) \bE \phi(X). \label{17a}
\end{equation}
The WLSI takes the form
\begin{align}
-\varLambda+\log(L_1L_2)\bE\phi(X)&\geq(\cos
\alpha)^2\frac{\varPhi(b_1)-\varPhi(a_1)}{L_1}\log \biggl(\frac{L_1}{\cos\alpha} \biggr)
\nonumber
\\
&\quad +(\sin\alpha)^2\frac{\varPhi(b_2)-\varPhi(a_2)}{L_2}\log \biggl(
\frac{L_2}{\sin\alpha} \xch{\biggr),}{\biggr)} \label{17b}
\end{align}
where
\begin{align}
\varLambda&=\frac{\log L_1}{L_2} \bigl[\varPhi(C_1)-
\varPhi(C_2) \bigr]+\frac{1}{L_1L_2} \Biggl[\int
_A^{C_1}\phi(x) (x-A)\log(x-A){\rm d}x
\nonumber
\\
&\quad +\int_{C_2}^B\phi(x) (B-x)\log(B-x){\rm d}x
\Biggr],
\nonumber
\\
A&=a_1+a_2,\ B=b_1+b_2.
\label{17c}
\end{align}
Finally, define $\varPhi^*(x)=\int_0^xu\phi(u){\rd u}$ and note that
\begin{align}
\bE\phi(X)&=\frac{1}{L_1L_2} \bigl[\varPhi^*(C_1)-\varPhi^*(A)-
\varPhi^*(B)+\varPhi^*(C_2) \bigr]
\nonumber
\\
&\quad -A \bigl[\varPhi(C_1)-\varPhi(A) \bigr]+L_1
\bigl[\varPhi(C_2)-\varPhi(C_1)\bigr]+B\bigl[\varPhi(B)-
\varPhi(C_2) \bigr]. \label{17d}
\end{align}
\end{example}

\section{The WLSI for the WF close to a constant}

\begin{proposition}
Let $d=1$, $X_i\sim{}$N$(\mu_i,\sigma_i^2), i=1,2$ be independent and
$X=X_1+X_2\sim{}$N$(\mu_1+\mu_2, \sigma_1^2
+\sigma_2^2)$. Suppose that WF $x\to\phi(x)$ is twice continuously
differentiable and
\begin{equation}
\bigl\llvert\phi''(x) \bigr\rrvert\leq\epsilon\phi(x),
\qquad
\bigl\llvert\phi(x)-\bar\phi \bigr\rrvert\leq\xch{\epsilon,}{\epsilon} \label{18}
\end{equation}
where $\epsilon>0$ and $\bar\phi >0$ are constants. Then there exists
$\epsilon_0>0$ such that for any WF
$\phi$ satisfying (\ref{18}) with $0<\epsilon<\epsilon_0$
WLSI holds true. Hence, checking of the WEPI is reduced to condition
(\ref{14}).
\end{proposition}

For a RV $\mathbf Z$, $\gamma>0$ and independent Gaussian RV $\mathbf
N\sim{}$N$(\mathbf0,{\bf I}_d)$ define
\begin{equation}
M(\mathbf Z;\gamma)=\bE \bigl[\bigl\| \mathbf Z-\bE[\mathbf Z\vert\mathbf Z\sqrt{
\gamma}+\mathbf N]\bigr\|^2 \xch{\bigr],}{\bigr]}
\end{equation}
where $\|.\|$ stands for the Euclidean norm. According to \cite{VG,KS2}
the differential entropy
\begin{equation}
h(\mathbf Z)=h(\mathbf N)+\frac{1}{2}\int_0^{\infty}
\bigl[M(\mathbf Z;\gamma)-{\bf1}_{\{
\gamma<1\}} \bigr]{\rm d}\gamma.
\label{19}
\end{equation}
For $\mathbf Z=\mathbf Y_1, \mathbf Y_2, \mathbf X_1+\mathbf X_2$
assume the following conditions
\begin{equation}
\bE \bigl[ \bigl\llvert\log f_\mathbf Z(\mathbf Z) \bigr\rrvert \bigr]<
\infty, \bE \bigl[\|\mathbf Z\|^2 \bigr]<\infty \label{20}
\end{equation}
and the uniform integrability: for independent $\mathbf N,\mathbf
N'\sim{}$N$(\mathbf0,\mathbf I)$ and any $\gamma>0$ there exist an
integrable function
$\xi(\mathbf Z,\mathbf N)$ such that
\begin{equation}
\biggl\llvert\log\bE \biggl[f_\mathbf Z \biggl(\mathbf Z+
\frac{\mathbf N-\mathbf
N'}{\sqrt{\gamma}}\vert\mathbf Z,\mathbf N \biggr) \biggr] \biggr\rrvert\leq\xi(
\mathbf Z,\mathbf N). \label{20a}
\end{equation}

\begin{theorem}
Let $d=1$ and assume conditions (\ref{20}), (\ref{20a}). Let $\gamma_0$
be a point of continuity of $M(Z;\gamma), Z=Y_1, Y_2, X_1+X_2$.
Suppose that there exists $\delta>0$ such that
\begin{equation}
M(X_1+X_2;\gamma_0)\geq M(Y_1,
\gamma_0) (\cos\alpha)^2+M(Y_2;\gamma
_0) (\sin\alpha)^2+\delta. \label{21}
\end{equation}
Suppose that for some $\bar\phi >0$ the WF satisfies
\begin{equation}
\bigl\llvert\phi(x)-\bar\phi \bigr\rrvert<\epsilon. \label{22}
\end{equation}
Then there exists $\epsilon_0=\epsilon_0(\gamma_0,\delta, f_1,f_2)$
such that for any WF satisfying
(\ref{22}) with $\epsilon<\epsilon_0$ the WLSI holds true.
\end{theorem}

\begin{proof} For a constant WF $\bar\phi $, the following inequality is
valid (see \cite{KS2}, Lemma~4.2 or \cite{VG},
Eqns (9) and (10))
\begin{equation}
(\cos\alpha)^2h^\rw_{\bar\phi }(Y_1)+(\sin
\alpha)^2h^\rw_{\bar
\phi }(Y_2)\leq
h^\rw_{\bar\phi }(Y_1\cos\alpha+Y_2\sin
\alpha). \label{22a}
\end{equation}
However, in view of Theorem~4.1 from \cite{KS2}, the representation
(\ref{19}) and inequality (\ref{21}) imply under conditions (\ref{20})
and (\ref{20a}) a stronger inequality
\begin{equation}
(\cos\alpha)^2h^\rw_{\bar\phi }(Y_1)+(\sin
\alpha)^2h^\rw_{\bar
\phi
}(Y_2)+c_0
\delta\leq h^\rw_{\bar\phi }(Y_1\cos
\alpha+Y_2\sin \alpha). \label{22b}
\end{equation}
Here $c_0>0$ and the term of order $\delta$ appears from integration in
(\ref{19}) in a neighbourhood of the continuity point $\gamma_0$.
Define $\varphi^*(x)=|\phi(x)-\bar\phi |$. It is easy to check that
\begin{equation}
h^\rw_{\varphi^*}(Z)<c_1\epsilon,
\quad
Z=X_1, X_2, X_1+X_2. \label{22c}
\end{equation}
From ({\ref{22b}}) and (\ref{22c}) we obtain that for $\epsilon$
small enough
\begin{equation}
(\cos\alpha)^2h^\rw_{\phi}(Y_1)+(\sin
\alpha)^2h^\rw_{\phi
}(Y_2)\leq
h^\rw_{\phi}(Y_1\cos\alpha+Y_2\sin
\alpha), \label{22d}
\end{equation}
i.e., the WLSI holds true.
\end{proof}

As an example, consider the case where RVs $X_1, X_2$ are normal and WF
$\phi\in C^2$.

\begin{proposition}
Let RVs $X_i\sim{}$N$(\mu_i,\sigma_i^2), i=1,2$ be independent, and
$X=X_1+X_2\sim{}$N$(\mu_1+\mu_2, \sigma_1^2+\sigma_2^2)$.
Suppose that WF $x\in\bbR\to\phi(x)\geq0$
is twice contiuously
differentiable and slowly varying in the sense that $\forall x$,
\begin{equation}
\bigl\llvert\phi''(x) \bigr\rrvert\leq\epsilon\phi(x),
\qquad
\bigl\llvert\phi(x)-\bar\phi \bigr\rrvert<\xch{\epsilon,}{\epsilon} \label{23}
\end{equation}
where $\epsilon>0$ and $\bar\phi >0$ are constants. Then there exists
$\epsilon_0=\epsilon_0(\mu_0,\mu_1,\sigma_0^2,\sigma_2^2)>0$
such that for any $0<\epsilon\leq\epsilon_0$, the WLSI (\ref{15}) with
the WF $\phi$ holds true.
\end{proposition}

\begin{proof} Let $\alpha$ be as in (\ref{13a}); to check (\ref{15}),
we use Stein's formula: for $Z\sim{}$N$(0,\sigma^2)$
\begin{equation}
\bE \bigl[Z^2\phi(Z) \bigr]=\sigma^2\bE \bigl[\phi(Z)
\bigr]+\sigma^4\bE \bigl[\phi''(Z)
\bigr]. \label{24}
\end{equation}
Owing the inequality $|\phi(x)-\bar\phi |<\epsilon$ we have
\begin{align}
\alpha<\alpha_0&= \tan^{-1} \bigl( \exp \bigl((\bar\phi + \epsilon )^2\bigl[h_+(X_2)-h_{-}(X_1)\bigr]\nonumber\\[6pt]
&\quad -(\bar\phi -\epsilon)^2\bigl[h_+(X_1)-h_{-}(X_2)\bigr] \bigr)  \bigr). \label{25}
\end{align}
Here
\begin{equation}
h_{\pm}(X_i)=-\bE \bigl[{\B1} \bigl(X_i\in
A^i_{\pm} \bigr)\log f^{No}_{X_i}(X_i)
\bigr], \quad i=1,2. \label{36}
\end{equation}
and
\begin{equation}
A^i_+= \bigl\{x\in\bR: f^{No}_i(x)<1 \bigr\},
\qquad
A^i_{-}= \bigl\{x\in\bR: f^{No}_i(x)>1\bigr\},
\quad
i=1,2. \label{27}
\end{equation}
Evidently, under conditions $|\phi'(x)|, |\phi''(x)|<\epsilon\phi(x)$
we have that
$\alpha_0<\frac{\pi}{2}-\epsilon$ and $0<\epsilon< (\sin\alpha
)^2, (\cos
\alpha)^2<1-\epsilon<1$.
We claim that inequality (\ref{17}) is satisfied with $\phi$ replaced
by $\bar\phi $ and added $\delta>0$:
\begin{align}
&\log \bigl[2\pi\bigl(\sigma_1^2+\sigma_2^2
\bigr) \bigr]\bar\phi +\frac{\log e}{\sigma_1^2+\sigma_2^2}\bE\bigl[X^2\bigr]\bar\phi
\nonumber
\\[6pt]
&\quad \geq(\cos\alpha)^2 \biggl[\log \biggl(\frac{2\pi\sigma_1^2}{(\cos\alpha)^2}
\biggr) \biggr]\bar\phi +\frac{(\cos\alpha)^2\log e}{\sigma_1^2}\bE\bigl[X_1^2
\bigr]\bar\phi
\nonumber
\\[6pt]
&\qquad +(\sin\alpha)^2 \biggl[\log \biggl(\frac{2\pi\sigma_2^2}{(\sin\alpha)^2}\biggr)
\biggr]\bar\phi +\frac{(\sin\alpha)^2\log e}{\sigma_2^2}\bE\bigl[X_2^2\bigr]
\bar\phi+\delta. \label{28}
\end{align}
Here $\delta>0$ is calculated through $\epsilon$ and increases to a
limit $\delta_0>0$ as $\epsilon\downarrow0$.
Indeed, strict concavity of $\log y$ for $y\in[0,\frac{2\pi\sigma
_1^2}{(\cos\alpha)^2}\vee\frac{2\pi\sigma_2^2}{(\sin\alpha)^2}]$
implies that
\begin{align}
\log \bigl[2\pi\bigl(\sigma_1^2+\sigma_2^2
\bigr) \bigr]&\geq(\cos\alpha)^2 \biggl[\log \biggl(\frac{2\pi\sigma_1^2}{(\cos\alpha)^2}
\biggr) \biggr]
\nonumber
\\[6pt]
&\quad +(\sin\alpha)^2 \biggl[\log \biggl(\frac{2\pi\sigma_2^2}{(\sin\alpha)^2}\biggr)
\biggr]+\delta. \label{29}
\end{align}
On the other hand,
\begin{equation}
\frac{1}{\sigma_1^2+\sigma_2^2}\bar\phi \bE \bigl[X^2 \bigr]=\frac{ (\cos
\alpha
)^2}{\sigma_1^2}
\bar \phi \bE \bigl[X_1^2 \bigr]+\frac{(\sin\alpha)^2}{\sigma
_2^2}\bar
\phi \bE \bigl[X_2^2 \bigr]. \label{30a}
\end{equation}
Combining (\ref{29}) and (\ref{30a}) one gets (\ref{28}).
Now, to check (\ref{15}) with WF $\phi$, in view of (\ref{28}) it
suffices to verify
\begin{align}
&\log \bigl[2\pi\bigl(\sigma_1^2+\sigma_2^2
\bigr) \bigr] \bigl[\bE\phi(X)-\bar\phi\bigr]+\frac{\log e}{\sigma_1^2+\sigma_2^2}\bE
\bigl[X^2\bigl(\phi(X)-\bar\phi \bigr)\bigr]
\nonumber
\\
&\quad -(\cos\alpha)^2 \biggl[\log \biggl(\frac{2\pi\sigma_1^2}{(\cos\alpha)^2}\biggr)
\biggr] \bigl[\bE\phi(X_1)-\bar\phi \bigr]+\frac{(\cos\alpha)^2\log e}{\sigma_1^2}\bE
\bigl[X_1^2\bigl(\phi(X_1)-\bar\phi \bigr)
\bigr]
\nonumber
\\
&\quad -(\sin\alpha)^2 \biggl[\log \biggl(\frac{2\pi\sigma_2^2}{(\sin\alpha)^2}\biggr)
\biggr] [\bE\phi(X_2)-\bar\phi )\,{+}\frac{(\sin\alpha)^2\log e}{\sigma_2^2}\bE
\bigl[X_2^2\bigl(\phi(X_2)\,{-}\,\bar\phi \bigr)
\bigr]{<}\,\delta. \label{30}
\end{align}
We check (\ref{30}) by a brute force, claiming that each term in (\ref
{30}) has the absolute value
$<\delta/6$ when $\epsilon$ is small enough. For the terms containing
$\bE[\phi(Z)-\bar\phi ]$, $Z=X,X_1X_2$,
this follows since $|\phi-\bar\phi |<\epsilon$.
For the terms containing factor $\bE[Z^2(\phi(Z)-\bar\phi )]$, we use
Stein's formula (\ref{24}) and the condition that
$|\phi''(x)|\leq\epsilon\phi(x)$.
\end{proof}

Similar assertions can be established for other examples of PDFs
$f_1(x)$ and $f_2(x)$, i.e. uniform, exponential, Gamma,
Cauchy, etc.

\section{A weighted Fisher information inequality}

Let $\mathbf Z= (\mathbf X,\mathbf Y)$ be a pair of independent RVs
$\mathbf X$ and $\mathbf Y\in\bbR
^d$, with sample values
$\mathbf z= (\mathbf x, \mathbf y) \in\bbR^d \times\bbR^d$ and
marginal PDFs $f_1(\mathbf x
,\utheta), f_2(\mathbf y,\utheta)$, respectively.
Let $f_{\mathbf Z|\mathbf X+\mathbf Y}(\mathbf x,\mathbf y|\bu)$ stand
for the conditional PDF as
\begin{equation}
f_{\mathbf Z\llvert\mathbf X+\mathbf Y}(\mathbf x,\mathbf y\rrvert\bu)=\frac
{f_1(\mathbf x)f_2(\mathbf y)\B1(\mathbf x+\mathbf y=\bu
)}{\int_{\bbR^n}f_1(\mathbf v)f_2(\bu-\mathbf v)\rd\mathbf v}.
\label{70}
\end{equation}
Given a WF $\mathbf z= (\mathbf x, \mathbf y) \in\bbR^d \times\bbR
^d \mapsto\phi
(\mathbf z) \geq0$, we employ the following reduced
WFs:
\begin{align}
\varphi(\bu)&=\int\phi(\mathbf v, \bu-\mathbf v)f_{\mathbf Z|\mathbf X+\mathbf Y}(\mathbf v, \bu-
\mathbf v)\rd\mathbf v,
\nonumber
\\
\vp_1(\mathbf x)&=\int\phi(\mathbf x+\mathbf y,\mathbf
y)f_2(\mathbf y)\rd\mathbf y, \vp_2(\mathbf y)=\int\phi(
\mathbf x, \mathbf x+\mathbf y)f_1(\mathbf x)\rd\mathbf x.
\label{71_2}
\end{align}
Next, let us introduce the matrices $M_{\phi}$ and $G_{\phi}$:
\begin{align}
M_{\phi}&=\int\phi(\mathbf x,\mathbf y)f_1(\mathbf
x)f_2(\mathbf y) \biggl(\frac{\partial\log f_1(\mathbf x)}{\partial\utheta} \biggr)^{\rT}
\biggl(\frac{\partial\log f_2(\mathbf x)}{\partial\utheta} \biggr)\B1\bigl(f_1(\mathbf
x)f_2(\mathbf y)>0\bigr)\rd\mathbf x\rd\mathbf y,
\nonumber
\\
G_{\phi}&= \bigl(J^\rw_{\varphi_1}(\mathbf X)
\bigr)^{-1} M_{\phi}\bigl(J^{\rw}_{\phi_2}(
\mathbf Y) \bigr)^{-1}. \label{72_2}
\end{align}

Note that for $\phi\equiv1$ we have $ M_{\phi}= G_{\phi}=0$ and the
classical Fisher information inequality emerges (cf. \cite{Z}).
Finally, we define
\begin{align}
\varXi&:=\varXi_{\varphi_1,\varphi_2}(\mathbf X, \mathbf Y)= M_{\phi}J^\rw_ {\varphi_1}(
\mathbf X)G_{\phi} (\mathbf I-M_{\phi}G_{\phi}
)^{-1}M_{\phi}\bigl[G_{\phi}J^\rw_{\varphi_2}(
\mathbf Y)G_{\phi}-J^\rw_{\varphi_1}(\mathbf X) \bigr]
\nonumber
\\
&\quad +G_{\phi} (\mathbf I-M_{\phi}G_{\phi}
)^{-1} M_{\phi}G_{\phi} J^\rw_{\vp_2}(
\mathbf Y) \bigl[M_{\phi}^{-1}- G_{\phi}
\bigr]-G_{\phi}J^\rw_{\varphi_2}(\mathbf
Y)G_{\phi}-G_{\phi}. \label{72_3}
\end{align}

\begin{theorem}[A weighted Fisher information inequality (WFII)]
Let $\mathbf X$ and $\mathbf Y$ be independent RVs.
Assume that $f^{(1)}_{\mathbf X}=\frac{\partial}{\partial\utheta
}f_{1}$ is
not a multiple of $f^{(1)}_{\mathbf Y}=\frac{\partial}{\partial
\utheta
}f_{2}$. Then
\begin{equation}
J^\rw_{\varphi}(\mathbf X+\mathbf Y)\leq (\mathbf
I-M_{\phi
}G_{\phi} ) \bigl[ \bigl(J^\rw_{\varphi_1}(
\mathbf X) \bigr)^{-1}+ \bigl(J^\rw_{\varphi
_2}(\mathbf
Y) \bigr)^{-1}-\varXi_{\varphi_1,\varphi_2}(\mathbf X, \mathbf Y)
\bigr]^{-1}. \label{90}
\end{equation}
\end{theorem}

\begin{proof} We use the same methodology as in Theorem~1 from \cite
{Z}. Recalling Corollary~4.8,
(iii) in \cite{SSSK} substitute $\tP:= [1, 1]$. Therefore for
$\mathbf Z=
(\mathbf X,\mathbf Y)$, $J^\rw(\mathbf Z)$ is an $m \times m$ matrix
\begingroup
\abovedisplayskip=7.5pt
\belowdisplayskip=7.5pt
\begin{equation}
\bigl(\tJ^{\rm w}_{{\theta}}(\BZ) \bigr)^{-1}=\lleft(
\begin{array}{cc}
\tJ^{\rm w}_{\vp_1}(\BX) & M_{\phi} \\
M_{\phi} & \tJ^{\rm w}_{\vp_2}(\BY) \\
\end{array} %
 \rright)^{-1}.
\end{equation}
Next, we need the following well-known expression for the inverse of a
block matrix
\begin{equation}
\label{inverse} %
\begin{pmatrix}\BC_{11}&\BC_{21}\\ \BC_{12}&\BC_{22}
\end{pmatrix} %
^{-1}=
\begin{pmatrix} \BC_{11}^{-1}+\BD_{12}\BD_{22}^{-1}\BD_{21}&-\BD
_{12}\BD
_{22}^{-1}\\ -\BD_{22}^{-1}\BD_{21}&\BD_{22}^{-1}
\end{pmatrix}, %
\end{equation}
where
\begin{eqnarray*}
&\BD_{22}=\BC_{22}-\BC_{21}\BC_{11}^{-1}\BC_{12}, \qquad \BD_{12}=\BC _{11}^{-1}\BC_{12}, \qquad \BD_{21}=\BC_{21}\BC_{11}^{-1},&\\
&\BC_{11}=\tJ^{\rw}_{\vp_1}(\BX), \qquad \BC_{22}=\tJ^{\rw}_{\vp_2}(\BY ),\quad \hbox{and}\quad \BC_{12}=\BC_{21}=\xch{M_{\vp}.}{M_{\vp},}&
\end{eqnarray*}
By using the Schwarz inequality, we derive
\begin{equation}
M_{\vp}^2\leq\tJ^{\rm w}_{\vp_1}(\BX)\;
\tJ^{\rm
w}_{\vp
_2}(\BY),\quad \hbox{or } M_{\vp}\;
G_{\vp} \leq\BI,
\end{equation}
with equality iff $ f^{(1)}_{\BX}(\mathbf x)=\frac{\partial
\log
f_1(\mathbf x)}{\partial\utheta}\propto\frac{\partial\log
f_2(\mathbf y
)}{\partial\utheta}= f^{(1)}_{\BY}(\mathbf y)$.

Define
\begin{align}
\label{def:delta} \delta&:=\tJ^{\rm w}_{\vp_2}(\BY)-
\tM_{\vp} \bigl(\tJ^{\rm w}_{\vp_1}(\BX)
\bigr)^{-1}M_{\vp}=(\tI-M_{\vp}\;G_{\vp} )
\tJ^{\rm w}_{\vp_2}(\BY)
\nonumber
\\
&\Rightarrow\delta^{-1}= \bigl(\tJ^{\rm w}_{\theta_2}(\BY)
\bigr)^{-1}(\tI-M_{\vp}\;G_{\vp} )^{-1}.
\end{align}
Thus, owing to the (\ref{inverse}), particularly for $\tP=[1,1]$, we
can write
\begin{align}
\tP \bigl(\tJ^{\rm w}_{{\vp}}(\BZ) \bigr)^{-1}
\tP^T&= \bigl(\tJ^{\rm w}_{\vp_1}(\BX)
\bigr)^{-1}+ \bigl(\tJ^{\rm w}_{\vp_1}(\BX)
\bigr)^{-1} M_{\vp}\;\delta^{-1}\; M_{\vp}
\bigl(\tJ^{\rm w}_{\vp_1}(\BX) \bigr)^{-1}
\nonumber
\\
&\quad - \bigl(\tJ^{\rm w}_{\vp_1}(\BX) \bigr)^{-1}
M_{\vp}\;\delta^{-1}-\delta^{-1}\; M_{\vp}
\; \bigl(\tJ^{\rm w}_{\vp_1}(\BX)\bigr)^{-1}+
\xch{\delta^{-1}.}{\delta^{-1}}
\label{eq:2.10}
\end{align}
Substituting (\ref{def:delta}), in above expression, we have
\begin{align}
&\tP \bigl(\tJ^{\rm w}_{{\vp}}(\BZ) \bigr)^{-1}\tP^T\nonumber\\
&\quad = \bigl(\tJ^{\rm w}_{\vp_1}(\BX)\bigr)^{-1}+ G_{\vp} (\mathbf I- M_{\vp}\;G_{\vp} )^{-1} M_{\vp}\; \bigl(\tJ^{\rm w}_{\vp_1}(\BX) \bigr)^{-1}\nonumber\\
&\qquad - G_{\vp} (\mathbf I- M_{\vp}\; G_{\vp})^{-1}-\bigl(\tJ^{\rm w}_{\vp_2}(\BY)\bigr)^{-1} (\BI- M_{\vp}\; G_{\vp} )^{-1}M_{\vp}\; \bigl(\tJ^{\rm w}_{\vp_1}(\BX)\bigr)^{-1}\nonumber\\
&\qquad + \bigl(\tJ^{\rm w}_{\vp_2}(\BY) \bigr)^{-1} (\mathbf I- M_{\vp}\; G_{\vp} )^{-1}\nonumber\\
&\quad = \bigl\{\bigl(\tJ^{\rm w}_{\vp_1}(\BX) \bigr)^{-1}(\mathbf I-M_{\vp}\;G_{\vp} )\nonumber\\
&\qquad + G_{\vp} (\mathbf I-M_{\vp}\; G_{\vp})^{-1} M_{\vp}\bigl(\tJ^{\rm w}_{\vp_1}(\BX)\bigr)^{-1} (\mathbf I- M_{\vp}\; G_{\vp})-G_{\vp}\nonumber\\
&\qquad - \bigl(\tJ^{\rm w}_{\theta_2}(\BY)\bigr)^{-1}(\mathbf I- M_{\vp}\; G_{\vp})^{-1}\tM_{\vp} \bigl(\tJ^{\rm w}_{\vp_1}(\BX)\bigr)^{-1} (\tI- M_{\vp}\; G_{\vp} )\nonumber\\
&\qquad + \bigl(\tJ^{\rm w}_{\vp_2}(\BY) \bigr)^{-1}\bigr\} (\mathbf I-M_{\vp}\;G_{\vp} )^{-1}.
\label{92}
\end{align}
\endgroup
Consequently by simplifying (\ref{92}), one yields
\begin{align}
&\tP \bigl(\tJ^{\rm w}_{{\vp}}(\BZ) \bigr)^{-1}
\tP^T
\nonumber
\\
&\quad = \bigl\{ \bigl(\tJ^{\rm w}_{\vp_1}(\BX)
\bigr)^{-1}+ \bigl(\tJ^{\rm w}_{\vp_2}(\BY)
\bigr)^{-1}+\varXi_{\vp_1,\vp_2}(\BX,\BY) \bigr\}(\mathbf I-
M_{\vp}\;G_{\vp} )^{-1}. \label{93}
\end{align}
By using Corollary~3.4, (iii) from \cite{SSSK}, we obtain the property
claimed in (\ref{90}):
\begin{equation*}
\tJ^{\rm w}_{\vp}(\BX+\BY)\leq \bigl\{ \bigl[ \bigl(\tJ
^{\rm w}_{\vp
_1}(\BX) \bigr)^{-1}+ \bigl(
\tJ^{\rm w}_{\vp_2}(\BY) \bigr)^{-1}+\varXi
_{\vp
_1,\vp_2}(\BX,\BY) \bigr] (\mathbf I-M_{\vp}\;G_{\vp}
)^{-1} \bigr\}^{-1}.
\end{equation*}
This concludes the proof.
\end{proof}

\begin{proposition}
Consider additive RV $\BZ=\BX+\BN_{\BSigma}$, such that $\BN
_{\BSigma
}\sim{\rm N}(\mathbf0,\BSigma)$ and $\BN_{\BSigma}$ is independent
of $\BX$.
Introduce matrices
\begin{align}
V_{\vp}(\BX\vert\BZ)&=\bbE \bigl[\vp\; \bigl(\BX-\bbE[\BX\vert\BZ]
\bigr)^{\rT}\bigl(\BX-\bbE(\BX\vert\BZ)\bigr) \bigr],
\nonumber
\\
E_{\vp}&=\bbE \bigl[\vp\; \bigl(\BZ-\bbE [\BX|\BZ]
\bigr)^{\rT}\bigl(\BX-\bbE [\BX|\BZ ] \bigr) \bigr],
\quad
\overline{E}_{\vp}=E_{\vp}+ E_{\vp}^{\rT}.
\label{eq:V and E}
\end{align} %

The WFIM of RV $\BZ$ can be written as
\begin{equation}
\label{WFIM:Y} \tJ^{\rw}_{\vp}(\BZ)= \bigl(
\BSigma^{-1} \bigr)^{\rT} \bigl\{\bbE \bigl[\vp\;\BN
_{\BSigma}^{\rT}{\BN_{\BSigma}} \bigr]+\overline{E}_{\vp}-
V_{\vp
}(\BX|\BZ ) \bigr\} \BSigma^{-1}.
\end{equation}
\end{proposition}

\section{The weighted entropy power is a concave function}\label{concave}
\def\gam{\gamma}
Let $\mathbf Z=\mathbf X+\mathbf Y$ and $\mathbf Y\sim{}$N$(\mathbf0,
\sqrt{\gamma}\mathbf I_d)$.
In the literature, several elegant proofs, employing the Fisher information
inequality or basic properties of mutual information, have been
proposed in order
to prove that the entropy power (EP) is a concave function of $\gam$
\cite{D,Vi}.
We are interested in the {\it weighted entropy power} (WEP) defined as follows:
\begin{equation}
{\rm N}^{\rw}_{\vp}(\BZ):={\rm N}^{\rw}_{\vp}(f_\mathbf
Z)=\exp \biggl\{\frac{2\;
\hw_{\vp}(\BZ)}{d\;\bbE[\vp(\BZ)]} \biggr\}. \label{94}
\end{equation}

Compute the second derivative of the WEP
\begin{align}
\frac{\rd^2}{\rd\gam^2}\exp \biggl\{\frac{2}{d}\frac{\hwphi(\BZ)}{\bbE[\phi(\BZ)]} \biggr
\}&= \exp \biggl\{\frac{2}{d}\frac{\hwphi(\BZ)}{\bbE[\phi(\BZ)]} \biggr\}
\nonumber
\\
&\quad \times \biggl[ \biggl(\frac{2}{d}\frac{\rd}{\rd\gam}
\frac{\hwphi(\BZ)}{\bbE[\phi(\BZ)]} \biggr)^2+ \biggl(\frac{2}{d}
\frac{\rd^2}{\rd\gam^2}\frac{\hwphi(\BZ)}{\bbE[\phi(\BZ)]} \biggr) \biggr]
\nonumber
\\
&=\exp \biggl\{\frac{2}{d}\frac{\hwphi(\BZ)}{\bbE[\phi(\BZ)]} \biggr\} \biggl[\bigl(
\varLambda(\gam)\bigr)^2+\frac{\rd}{\rd\gam}\varLambda(\gam) \xch{\biggr],}{\biggr]}
\label{95}
\end{align}
where
\begin{equation}
\varLambda(\gamma)=\frac{2}{d}\frac{\rd}{\rd\gamma}
\frac
{\hwphi(\BZ)}{\bbE[\phi(\BZ)]}. \label{96}
\end{equation}
In view of (\ref{95}) the concavity of the WEP is equivalent to the inequality
\begin{equation}
\frac{\rd}{\rd\gamma} \bigl(\varLambda(\gam) \bigr)^{-1}\geq1.
\label{97}
\end{equation}

In the spirit of the WEP, we shall
present a new proof of concavity of EP.
Regarding this, let us apply the WFII (\ref{90}) to $\phi\equiv1$.
Then a straightforward computation gives
\begin{equation}
\label{deriv}\frac{\rd}{\rd\gamma}\frac{d}{{\rm
tr}\;
J(\BZ)} \geq 1.
\end{equation}

\begin{theorem}[A weighted De Bruijn's identity]
Let $\BX\sim f_\mathbf X$ be a RV in $\bbR^n$, with a PDF $f_\mathbf
X\in C^2$. For
a standard Gaussian RV $\BN\sim{\rm N}(\mathbf0,\mathbf I_d)$
independent of $\BX
$, and given $\gamma>0$, define the RV
$\BZ=\BX+\sqrt{\gam}\BN$ with PDF $f_\BZ$. Let $\bV_r$ be the
$d$-sphere of radius $r$ centered at the origin and having surface
denoted by $\bS_r$. Assume that for given WF $\vp$ and $\forall\gam
\in
(0,1)$ the relations
\begin{equation}
\int f_{\BZ}(\mathbf x) \bigl\llvert\ln f_{\BZ}(\mathbf x)\bigr\rrvert\rd \mathbf x< \infty,
\qquad
\int \bigl\llvert\nabla\; f_\BZ(\mathbf y)\ln f_\BZ(\mathbf y) \bigr\rrvert\rd\mathbf y<\infty
\label{98a}
\end{equation}
and
\begin{equation}
\lim_{r\rightarrow\infty} \int_{\bbS
_r}\vp(\mathbf y )\log f_Z(\mathbf y)
\bigl(\nabla f_Z(\mathbf y) \bigr)\rd S_r=\xch{0}{0,}
\label{98b}
\end{equation}
are fulfilled. Then
\begin{equation}
\frac{\rd}{\rd\gamma}\hwphi(\BZ)=\frac{1}{2}\; {\rm tr}\;
\tJ^{\rm w}_{\phi}(\BZ)-\frac{1}{2}\bbE \biggl[\phi\;
\frac
{\Delta f_Z(\BZ)}{f_Z(\BZ)} \biggr]+\frac{\mR(\gam)}{2}. \label{99}
\end{equation}
Here
\begin{equation}
{\mR}(\gamma)=\bbE \bigl[\nabla\vp\;\log f_\BZ(\BZ ) \bigl(\nabla
\log f_\BZ(\BZ) \bigr)^{\rT} \bigr].
\end{equation}
If we assume that $\vp\equiv1$, then the equality (\ref{99}) directly
implies (\ref{97}). Hence, the standard entropy power is a concave
function of $\gamma$.
\end{theorem}

Next, we establish
the concavity of the WEP when the WF is close to a constant.
\begin{theorem}
Assume conditions (\ref{98a}) and (\ref{98b}) and
suppose that $\forall\gamma\in(0,1)$
\begin{equation}
\label{step1}\frac{\rd}{\rd\gamma}\frac{d}{{\rm
tr}\;
J(\BZ)}\geq 1+\epsilon.
\end{equation}
Then $\exists\delta=\delta(\epsilon)$ such that any
WF $\phi$ for which $\exists
\bar\phi >0$:
$|\phi-\bar\phi |<\delta$, $|\nabla\;\phi|<\delta$ the WEP (\ref{94})
is a concave function of $\gamma$. Under the milder assumption
\begin{equation}
\label{step2}\frac{\rd}{\rd\gamma}\frac{d}{{\rm
tr}\;
J(\BZ)} \bigg\vert_{\gamma=0}\geq1+
\xch{\epsilon,}{\epsilon}
\end{equation}
the WEP is a concave function of $\gamma$ in a small neighbourhood of
$\gamma=0$.
\end{theorem}

\begin{proof} It is sufficient to check that
\begin{equation}
\frac{\rd}{\rd\gamma}\psi(\gamma )\geq1
\quad \mathrm{where}\ \psi(\gamma)=\biggl(\frac{2}{d}\frac{\rd}{\rd\gamma}\frac{h^\rw_{\phi}(\BZ)}{\bbE[\phi(\BZ)}]\biggr)^{-1}=\varLambda(\gamma)^{-1}.
\end{equation}
By a straightforward calculation
\begin{align}
\psi(\BZ)&=d\bigl({\bbE}\bigl[\phi(\BZ)\bigr]\bigr)^2 \biggl[
\frac{{\rm d}}{{\rm d}\gamma}h^\rw_{\phi}(\mathbf Z){\bbE}\bigl[\phi(\BZ)
\bigr]-h^\rw_{\phi}(\BZ)\frac{{\rm d}}{{\rm d}\gamma}{\bbE}\phi(\BZ)
\xch{\biggr]^{-1},}{\biggr]^{-1}.}
\nonumber
\\
\frac{{\rm d}}{{\rm d}\gamma}h^\rw_{\phi}(\BZ)&=\frac{1}{2}{\rm
tr}J^w_{\phi}(\BY)-\frac{1}{2}\frac{{\rm d}}{{\rm d}\gamma}{\bbE
}\bigl[\phi(\BY)\bigr]+\frac{1}{2}\mR(\gamma).
\end{align}
These formulas imply
\begin{align}
&\psi(\gamma)=\frac{d}{{\rm tr}\;J_{\phi}(\BZ)}+o(\delta).\nonumber\\
&{\rm as}\;\;1-\delta< {\bbE}\bigl[\phi(\BZ)\bigr]<1+\delta,
\quad
\bigl\llvert{\rm tr}\;J^\rw_{\phi}(\BZ)-{\rm tr}\;J(\BZ)\bigr\rrvert<\delta\;{\rm tr}\;J(\BZ).\label{100}
\end{align}
Next,
\begin{equation}
\frac{\rm d}{{\rm d}\gamma}{\bbE} \bigl[\phi(\BZ) \bigr]=\frac
{1}{2}\int\phi (y)
\Delta f_\BZ(\mathbf y)\rd\mathbf y
\end{equation}
and using the Stokes formula one can bound this term by $\delta$.
Finally, $|\mR(\gamma)|\leq\delta$ in view of (\ref{98b}), which leads
to the claimed result.
\end{proof}

\section{Rates of weighted entropy and information}
\label{Rates}

This section follows \cite{SS}.
The concept of a {\it rate} of the WE or WDE emerges when we work with outcomes
in a context of a discrete-time random process (RP):
\begin{equation}
\label{eq100}h^\rw_{\phi_n} (\bp_n)=-\bbE
\phi_n \bigl(\mathbf X _0^{n-1} \bigr)\log\bp
_n \bigl(\mathbf X_0^{n-1} \bigr):= \bbE
I^\rw_{\phi_n} \bigl(\mathbf X_0^{n-1}
\bigr).
\end{equation}
Here the WF $\phi_n$ is made
dependent on $n$: two immediate cases are where (a) $\phi_n(\mathbf x_1^n)
=\sum_{j=0}^n\psi(x_j)$ and (b) $\phi_n(\mathbf x_1^n)=\prod_{j=0}^n\psi(x_j)$ (an additive and multiplicative
WF, respectively). Next, $\mathbf X_0^{n-1}
=(X_0,\ldots,X_{n-1})$ is a random {\it string} generated by an RP.
For simplicity, let us focus
on RPs taking values in a finite set $\cX$. Symbol $\bbP$ stands for
the probability measure of $\mathbf X$, and $\bbE$ denotes the expectation
under $\bbP$.
For an RP with IID values, the
joint probability of a sample $\mathbf x_0^{n-1}=(x_0,\ldots,x_{n-1})$ is
$\bp_n(\mathbf x_0^{n-1})=\prod_{j=0}^{n-1}p(x_j)$,
$p(x)=\bbP(X_j=x)$ being the probability of an individual outcome
$x\in
\cX$. In the case of a Markov chain,
$\bp_n(\mathbf x_0^{n-1})=\lam(x_0)\prod_{j=1}^n
p(x_{j-1},x_j)$. Here $\lam(x)$ gives an initial distribution and
$p(x,y)$ is the transition probability on
$\cX$; to reflect this fact, we will sometimes use the notation $h^\rw
_{\phi_n} (\bp_n,\lam)$.
The quantity
\[
I^\rw_{\phi_n} \bigl(\mathbf x_0^{n-1}
\bigr):=-\phi_n \bigl(\mathbf x_0^{n-1} \bigr)
\log \bp_n \bigl(\mathbf x_0^{n-1} \bigr)
\]
is interpreted as a {\it weighted information}
(WI) contained in/conveyed by outcome $\mathbf x_0^{n-1}$.

In the IID case, the WI and WE admit the following representations. Define
$S(p)=-\bbE [\log p(X) ]$ and $H^\rw_\psi=-\bbE [\psi
(X)\log
p(X) ]$ to be the SE
and the WE, of the one-digit distribution (the capital letter is used
to make it distinct from $h^\rw_{\phi_n}$,
the multi-time WE).

(A) For an additive WF:
\begin{equation}
\label{eq101}I^\rw_{\phi_n} \bigl(\mathbf
x_0^{n-1} \bigr)=-\sum_{j=0}^{n-1}
\psi(x_j)\sum_{l=0}^{n-1}\log
\xch{p(x_l)}{p(x_l),}
\end{equation}
and
\begin{equation}
\label{eq102}h^\rw_{\phi_n}(\bp_n)=n(n-1)S(p)\bbE
\bigl[\psi (X) \bigr]+nH^\rw _\psi(p):=n(n-1)
\rA_0+n\rA_1.
\end{equation}

(B) For a multiplicative WF:
\begin{equation}
\label{eq103}I^\rw_{\phi_n} \bigl(\mathbf
x_0^{n-1} \bigr)=-\prod_{j=0}^{n-1}
\psi(x_j) \sum_{l=0}^{n-1}\log
\xch{p(x_l)}{p(x_l);}
\end{equation}
and
\begin{equation}
\label{eq104}h^\rw_{\phi_n}(\bp_n)=nH^\rw_\psi
(p) \bigl[\bbE\phi (X) \bigr]^{n-1}:=\rB_0^{n-1}
\times n\xch{\rB_1.}{\rB_1,}
\end{equation}
The values $\rA_0$, $\rB_0$ and their analogs in a general situation
are referred to as {\it primary} rates, and $\rA_1$, $\rB_1$ as {\it
secondary} rates.

\renewcommand{\thesubsection}{\thesection.\Alph{subsection}}
\subsection{WI and WE rates for asymptotically additive WFs}

Here we will deal with a stationary RP
$\mathbf X=(X_j,j\in\bbZ)$ and use the above
notation $\bp_n(\mathbf x_0^{n-1})=\bbP(\mathbf X_0^{n-1}=\mathbf
x_0^{n-1})$ for the
joint probability.
We will refer to
the limit present in
the Shannon--McMillan--Breiman (SMB) theorem (see, e.g., \cite{CT,KS1}) taking place for an ergodic RP:
\begin{equation}
\label{eq:SMBGen}\lim_{n\to\infty} \biggl[-\frac
{1}{n}\log \bp_n
\bigl(\BX_0^{n-1} \bigr) \biggr] =-\bbE\log\bbP
\bigl(X_0|\BX_{-\infty}^{-1} \bigr):=S,\quad \bbP
\hbox{-a.s.}
\end{equation}
Here $\bbP(y|\mathbf x_{-\infty}^{-1})$ is the conditional PM/DF for $X_0=y$
given $\mathbf x_{-\infty}^{-1}$,
an infinite past realization of $\BX$. An assumption upon WFs $\phi_n$
called {\it asymptotic additivity} (AA) is that
\begin{equation}
\label{eq:genAdd}\lim_{n\to\infty}\frac
{1}{n}\phi _n \bigl(\BX
_0^{n-1} \bigr)=\alpha,\quad \bbP\hbox{-a.s.}\; \hbox{and/or in\ $\rL_2(\bbP)$}.
\end{equation}
Eqns \eqref{eq:SMBGen}, \eqref{eq:genAdd} lead to the identification of
the primary rate:
$\rA_0=\alpha S$.

\begin{theorem}\label{thm10.1}
Given an ergodic
RP $\mathbf X$, consider
the WI $I^\rw_{\phi_n}(\BX_0^{n-1})$ and the WE $H^\rw_{\phi
_n}(\bp_n)$
as defined in \eqref{eq101}, \eqref{eq102}. Suppose that convergence
in \eqref{eq:genAdd} holds $\bbP$-a.s. Then:

{\rm{(I)}} We have that
\begin{equation}
\label{eq110}\lim_{n\to\infty} \frac{I^\rw_{\phi_n}(\BX_0^{n-1})}{n^2}=\alpha S,\quad \bbP\hbox{-a.s.}\vadjust{\eject}
\end{equation}

{\rm{(II)}} Furthermore,
\begin{enumerate}
\item[{\rm{(a)}}] suppose that the
WFs $\phi_n$ exhibit convergence \eqref{eq:genAdd}, $\bbP$-a.s.,
with a finite $\alpha$, and $ |\phi_n(\BX_0^{n-1})/n |\leq c$
where $c$ is a
constant independent of $n$. Suppose also that convergence in Eqn
\eqref{eq:SMBGen}
holds true. Then we have that
\begin{equation}
\label{eq111}\lim_{n\to\infty}\frac{h^\rw_{\phi
_n}(\bp
_n)}{n^2}=\alpha S.
\end{equation}
\item[{\rm{(b)}}] Likewise, convergence in Eqn \eqref{eq111} holds true whenever
convergences \eqref{eq:genAdd} and \eqref{eq:SMBGen} hold\ $\bbP
$-a.s. and
$|\log\bp_n(\BX_0^{n-1})/n|\leq c$ where $c$ is a constant.

\item[{\rm{(c)}}] Finally, suppose that convergence in \eqref{eq:SMBGen} and
\eqref{eq:genAdd} holds in $\rL_2 (\bbP)$, with finite $\alpha$ and
$S$. Then again,
convergence in \eqref{eq111} holds true.
\end{enumerate}
\end{theorem}

\begin{example} Clearly, the condition of stationarity cannot be dropped.
Indeed, let $\phi_n(x_0^{n-1})=\alpha n$ be an additive WF
and $\BX$ be a (non-stationary) Gaussian process with covariances $C=\{
C_{\mathit{ij}}, i,j\in{\bf Z}_+^1\}$.
Let $f_n$ be a $n$-dimensional PDF of the vector $(X_1,\ldots,X_n)$. Then
\begin{equation}
h^w_{\phi_n}(f_n)=\frac{\alpha n}{2} \bigl[n
\log(2\pi e)+\log \bigl({\rm det}(C_n) \bigr) \bigr]. \label{10.10}
\end{equation}
Suppose that the eigenvalues $\lambda_1\leq\cdots\leq\lambda_j\leq
\cdots\leq\lambda_n$ of $C_n$
have the order $\lambda_j\approx cj$. Then by Stirling's formula
the second term in \eqref{10.10} dominates and the scaling of
$h^w_{\phi_n}(f_n)$ is $(n^2\log n)^{-1}$ instead of $n^{-2}$ as $n\to
\infty$.
\end{example}

Theorem~\ref{thm10.1}. can be considered as an analog of the SMB theorem for the
primary WE
rate in the case of an AA WF. A specification of the secondary rate
$\rA
_1$ is given in
Theorem~\ref{thm10.3} for an additive WF. The WE rates for multiplicative WFs
are studied in Theorem~\ref{thm10.4} for the case where $\mathbf X$
is a stationary ergodic Markov chain on $\cX$.

\begin{theorem}\label{thm10.3}
Suppose that $\phi_n(\mathbf x_0^{n-1})=\sum_{j=0}^{n-1}\psi
(x_j)$. Let
$\BX$ be a stationary RP with the property that $\forall$ $i\in\bbZ$
there exists the limit
\begin{align}
&\lim_{n\to\infty}\sum_{j\in\bbZ:\,\llvert j+i\rrvert\leq n}\bbE \bigl[
\psi(X_0)\log p^{(n+i+j)}\bigl(X_j|
\BX_{-n-i}^{j-1}\bigr) \bigr]
\nonumber
\\
&\quad =\sum_{j\in\bbZ}\bbE \bigl[\psi(X_0)\log p
\bigl(X_j|\BX_{-\infty}^{j-1}\bigr) \bigr]:=-
\xch{\rA_1}{\rA_1,} \label{eq:rA1id}
\end{align}
and the last series converges absolutely. Then
$\lim_{n\to\infty}{\frac{1}{n}}H^\rw_{\phi_n}(\bp
_n)=\rA_1$.
\end{theorem}

\subsection{WI and WE rates for asymptotically multiplicative WFs}

The WI rate is given in Theorem~\ref{thm10.3}. Here we use the condition of
asymptotic multiplicativity:
\begin{equation}
\label{eq:genMul} \lim_{n\to\infty} \bigl[\phi_n \bigl(\BX_0^{n-1}
\bigr) \bigr]^{1/n}=\beta, \quad \bbP\hbox{-a.s.}
\end{equation}
\setcounter{theorem}{3}
\begin{theorem}\label{thm10.4}
Given an ergodic RP $\BX$ with a probability distribution $\bbP$,
consider the
WI $I^\rw_{\phi_n}(\mathbf x_0^{n-1})=-\phi_n(\mathbf x_0^{n-1})\log
\bp_n(\mathbf x_0^{n-1})$.
Suppose that convergence
in \eqref{eq:genMul} holds $\bbP$-a.s. Then the following limit holds true:
\[
\lim_{n\to\infty}\frac{1}{n}\log I^\rw_{\phi_n} \bigl(\BX
_0^{n-1} \bigr)=\beta,\quad \bbP\hbox{-a.s.}
\]
\end{theorem}

\begin{theorem}\label{thm10.5}
Assume that $\phi(\mathbf x_0^{n-1})=\prod_{j=0}^{n-1}\psi
(x_j)$, with
$\psi(x)>0$, $x\in\cX$. Let $\mathbf X$ be a stationary Markov chain with
transition probabilities
$p (x,y)>0$, $x,y\in\cX$. Then, 
for all initial distribution $\lam$,
\begin{equation}
\lim_{n\to\infty}\frac{1}{n}\log h^\rw _{\phi
_n}(
\bp_n,\lam )=\rB_0.
\end{equation}
Here
\begin{equation}
\rB_0 =\log\mu
\end{equation}
and $\mu>0$ is the Perron--Frobenius eigenvalue of the matrix $\ttM
= (\psi(x)p(x,y) )$
coinciding with the norm of $\ttM$.
\end{theorem}

The secondary rate $\rB_1$ in this case is identified through the
invariant probabilities $\pi(x)$
of the Markov chain and the Perron--Frobenius eigenvectors of matrices
$\ttM$ and~$\ttM^\rT$.

\begin{example} Consider a stationary sequence $X_{n+1}=\alpha
X_n+Z_{n+1}, n\geq0$, where $Z_{n+1}\sim{}$N$(0,\sigma^2)$
are independent, and $X_0\sim{}$N$(0,c)$, $c=\frac{1}{1-\alpha^2}$. Then
\begin{align}
h^w_{\phi_n}(f_n)&=\frac{1}{2}\bbE \Biggl[
\prod_{j=0}^{n-1}\psi(X_j)
\biggl(X_0^2-2\alpha X_0X_1
\nonumber
\\
&\quad +\bigl(1+\alpha^2\bigr)X_1^2-2\alpha
X_1X_2+\bigl(1+\alpha^2\bigr)X_2^2-
\cdots
\nonumber
\\
&\quad +\bigl(1+\alpha^2\bigr)X_{n-2}^2-2
\alpha X_{n-2}X_{n-1}+X_{n-1}^2-2\log
\biggl(\frac{\sqrt{1-\alpha^2}}{(2\pi)^{n/2}} \biggr)\biggr) \Biggr].
\end{align}
Conditions of Theorem~\ref{thm10.5} may be checked under some restrictions on
the WF $\psi$, see \cite{SS}.
\end{example}

\section*{Acknowledgments}
The article was prepared within the framework of the Academic Fund
Program at
the National Research University Higher School of Economics (HSE) and
supported within the subsidy granted to the HSE
by the Government of the Russian Federation for
the implementation of the Global Competitiveness Programme.

%


%
\end{document}